\documentclass[11pt,a4paper]{article}

\usepackage[nottoc,numbib]{tocbibind} 
\usepackage[intlimits]{amsmath}
\usepackage{amstext,amsfonts,amssymb,amscd,amsthm,cite}
\usepackage[all]{xy} 

\usepackage{graphicx}
\usepackage{epstopdf} 

\usepackage{geometry} 
\newtheorem{theorem}{Theorem}[section]

\newtheorem{lemma}[theorem]{Lemma}
\newtheorem{proposition}[theorem]{Proposition}

\theoremstyle{definition}
\newtheorem{definition}[theorem]{Definition}

\theoremstyle{remark}
\newtheorem{remark}[theorem]{Remark}
\newtheorem*{remark*}{Remark}

\numberwithin{equation}{section}
\numberwithin{figure}{section}

\begin{document}

\title{On traces of Fourier integral operators
localized at a finite set of points
\thanks{The work is supported by RFBR grant NN $16$-$01$-$00373$ A.}
}
\author{\MakeUppercase{P. A. Sipailo}\thanks{RUDN University, Russia}}

\maketitle

\begin{abstract}\noindent
Given a smooth embedding of manifolds $i: X \hookrightarrow M$ and a Fourier integral operator $\Phi$ acting on $M$,
obtained by quantization of a canonical transformation,
consider its trace $i^!(\Phi)$ on $X$ (in the sense of relative theory).
We discuss the situation when $i^!(\Phi)$
has the form of a Fourier--Mellin operator and, in particular, is localized at a finite set of points.
\end{abstract}

\tableofcontents

\section*{Introduction}

In the present paper we deal with traces of quantized canonical transformations
(Fourier integral operators associated with graphs of canonical transformations).
The trace of an operator is a central concept of relative elliptic theory, that is,
a theory associated with a pair manifold--submanifold (see, e.g.,~\cite{NaSt10}).

Given a smooth embedding $i: X \hookrightarrow M$
of manifolds and an operator $A$ acting on $M$, one defines the \textit{trace of A}
by the formula
$$
    i^!(A) = i^* A \,  i_*,
$$
where $i^*$ and $i_*$ stand for the boundary and coboundary operators
induced by $i$. Here the boundary operator is just the restriction operator from
$M$ to $X$, and the coboundary operator acts in the opposite direction in a dual manner.
Thus, $i^!$ is an operation that takes an operator on the ambient manifold $M$ to some operator on
its submanifold $X$.

This operation has one remarkable property. Namely, given an operator $A$
coming from some well-known class of operators, its trace $i^!(A)$ may happen
to be an operator of a very different nature. For example, the trace of a differential operator
has the structure of a pseudodifferential operator.

The things become more interesting when one considers operators on $M$
which are not (pseudo)differential. Here we refer to the papers~\cite{SaSt28, SaSt31, SaSt32, Losch1}
which deal with constructions equivalent to traces of shift operators.
It was shown there that the trace of a shift operator can produce an operator
which is smoothing outside some subset of a manifold.
We say that such an operator is \textit{localized} at this subset.

More specifically, special operators
were introduced, which are localized at isolated points in $X$
(we call them the Fourier--Mellin operators due to their structure).
They naturally occur from transversal
intersections or from group actions with isolated fixed points
(these situations are studied in the above-mentioned papers).
In the present paper we are interested in the situation when Fourier--Mellin operators
arise as traces of quantized canonical transformations.
We hope it would be a first step to an uncharted area
of general theory describing the traces of general Fourier integral operators.

The crucial role of the coboundary operator (and the coboundary operator itself)
was discovered by B.~Yu.~Sternin in~\cite{Ster1, Ster8}.
The trace as a standalone notion was introduced in~\cite{NoSt1, NoSt2}.

The detailed exposition of Fourier--Mellin operators one can find in
~\cite{Losch1, Losch2}. We note that a similar construction originally comes from~\cite{Ster8}.

The author is grateful to B.~Yu.~Sternin and A.~Yu.~Savin for their vital help and support.

\section{Preliminaries}

\subsection{Trace of an operator associated with a submanifold}
Let $M$ be a closed smooth manifold and $X$ its submanifold
of codimension $\nu$. Let $i: X \hookrightarrow M$ be the corresponding smooth embedding.
Given an operator $D$ on $M$, we associate with it a new operator $i^!(D)$
on $X$, which we call the \textit{trace of $D$},
defined as composition
of $D$ and boundary and coboundary operators induced by embedding $i$.
Let us give an explicit construction (see also~\cite{NoSt1, NoSt2, NaSt10}).

The \textit{boundary operator} $i^*$
is the restriction operator,
which associates with a function defined on the ambient manifold $M$ its restriction to submanifold $X$,
\begin{equation}\label{defeq:defBOp}
    i^*: u \longmapsto u\rvert_X.
\end{equation}
The boundary operator is continuous in the spaces
$$
    i^*: H^{s}(M) \longrightarrow H^{s-\frac{\nu}{2}}(X), \quad s - \nu/2 > 0.
$$

The \textit{coboundary operator} $i_*$ is adjoint to the boundary operator $i^*$.
It acts by the formula
\begin{equation}\label{defeq:defCobOp}
    i_*: u \longmapsto u \otimes \delta_X,
\end{equation}
where $\delta_X$ stands for the Dirac delta-function on $M$ concentrated on submanifold $X$.
The coboundary operator is continuous in the spaces
$$
    i_*: H^{-s + \frac{\nu}{2}}(X) \longrightarrow H^{-s}(M), \quad s - \nu/2 > 0.
$$

The \textit{trace} $i^!(D)$ of an operator $D$ of order $\operatorname{ord} D = m$
acting manifold $M$ is defined by the formula
\begin{equation}\label{defeq:sled}
    i^!(D) = i^*  D i_*: H^s(X) \longrightarrow H^{s-m-\nu}(X).
\end{equation}
The trace $i^!(D)$ is continuous in the specified spaces when $s < 0$, $s-m-\nu > 0$.
In particular, its order $m$ must be $m < -\nu$ (thus, $m$ is necessarily negative).

The \textit{reduced trace} $i^!_0(D)$ is obtained from $i^!(D)$ by reducing it to an operator acting in $L^2$-spaces.
We define it by the formula
$$
    i^!_0(D) = \Delta_{X}^{\frac{s-m-\nu}{2}} \, i^!(D) \, \Delta_{X}^{-\frac{s}{2}}: L^2(X) \longrightarrow L^2(X),
$$
where $\Delta_X$ stands for the Laplace operator on $X$.

\subsection{Fourier--Mellin operators}\label{section:operatorFM}
The most important property of Fourier--Mellin operators (see~\cite{Losch1})
is the property of being localized at a fixed point.
Here we say that an operator $A$ is \textit{localized in a set $S$} if for any cut-off function
$\varphi$ vanishing on $S$ compositions $A \varphi$ and $\varphi A$ are compact operators.

Let $x_0$ be a fixed point in manifold $X$. Fix its small neighbourhood, and let $x$ be coordinate functions
in this neighbourhood, such that $x_0$ is represented by the equation $x = 0$.
An operator $A$ on manifold $X$ is called the \textit{Fourier--Mellin operator} if
it is localized at $x_0$, and in $x$-coordinates has the following form
\begin{equation}\label{defeq:FM}
    A = \varphi(x) \, \mathcal{F}^{-1}_{p \rightarrow x}
        \, \psi(p) \,
        \mathcal{M}^{-1}_{\zeta \rightarrow r_p}
            \, K_\gamma(\zeta) \,
        \mathcal{M}_{r_p \rightarrow \zeta}
        \, \psi(p) \,
        \mathcal{F}_{x \rightarrow p} \, \varphi(x).
\end{equation}

Let us explain this formula. Starting from manifold $X$, one makes localization
in a coordinate neighbourhood of the fixed point with the help of multiplication
by a cut-off function $\varphi(x)$ such that
$\varphi(0) = 1$ in a neighbourhood of zero and $\varphi(x) = 0$ at infinity.
Then the Fourier transform $\mathcal{F}_{x \rightarrow p}$ takes us to the dual space,
where $p$ stands for Fourier coordinates. Then one cuts out the origin with the help of
a cut-off function $\psi(p)$ such that $\psi(p) = 0$ in a neighbourhood of zero and $\psi(p) = 1$ at infinity.
Then one performs the Mellin transform $\mathcal{M}_{r_p \rightarrow \zeta}$
with respect to the radial variable $r_p$, where we use spherical coordinates
$$
    p \longmapsto r_p \omega_p, \quad r_p \in \mathbb{R}_+, \, \omega_p \in \mathbb{S}^{n-1}.
$$
Next, $K_\gamma(\zeta)$ is an operator-valued function (of complex variable),
whose values are integral operators on the sphere of dimension $n-1$, satisfying the following conditions:
\begin{enumerate}
    \item $K_\gamma(\zeta)$ is analytic on the vertical line (further on we call it the \textit{weight line})
    \begin{equation}\label{defeq:weightLine}
        \Gamma_{\gamma} = \{\zeta \in \mathbb{C} \; | \; \operatorname{Re}(\zeta) = \gamma\}, \quad \gamma \in \mathbb{R}.
    \end{equation}
    \item Given $\zeta \in \Gamma_\gamma$, the value $K_\gamma(\zeta)$ is a compact operator.
    \item ${\left\| K_\gamma(\zeta) \right\|} \to 0$ as ${\lvert \zeta \rvert} \to \infty$ whenever $\zeta \in \Gamma_\gamma$.
\end{enumerate}
After multiplying by $K_\gamma(\zeta)$ in~\eqref{defeq:FM} one applies inverse transformations and returns to the original space.
Thus,~\eqref{defeq:FM} gives a well-defined operator on $X$.

It was shown in~\cite{Losch1} that
 that formula~\eqref{defeq:FM} determines a unique operator up to smoothing operators
independently of the choice of cut-off functions $\varphi$ and $\psi$.
It is continuous in the spaces $H^s(X) \rightarrow H^s(X)$ whenever
\begin{equation}\label{eq:weightGaForHs}
    \gamma = s + \dim X/2.
\end{equation}

The function $K_\gamma (\zeta)$ is called the \textit{symbol} of operator~\eqref{defeq:FM}.
There is a natural definition of ellipticity for Fourier--Mellin operators
given in terms of its symbol, and the corresponding finiteness theorem.
The index formula is also established (see~\cite{Losch2}).

Note that Fourier--Mellin operators can be associated with a finite set of points in an obvious way.

\subsection{Fourier integral operators}\label{subsec:operatorFIO}
For a more detailed exposition of Fourier integral operators
(and quantized canonical transformations)
given in terms of Maslov canonical operator see, e.g.,~\cite{MSS1, NOSS1}.

By a Fourier integral operator (FIO) $\Phi = \Phi(L,a)$ associated with a Lagrangian manifold $L$
and an amplitude $a$ we mean an integral operator whose Schwartz kernel is given
by a canonically represented function, that is, a function (more preciously, a distribution)
obtained as a result of applying a Maslov canonical operator
associated with $L$ (see~\cite{Mas1}) to amplitude $a$.
Given a manifold $M$, FIO $\Phi(L,a)$ on $M$ can be locally presented in the following way.

Let $L$ be a Lagrangian submanifold embedded into
$T^*_0(M \times M) = T^*(M \times M) \setminus \{0\}$, and $(w_I,w'_{I'})$
be canonical coordinates on $L$ given in the form
\begin{equation*}
    (w_I, w'_{I'})
    =
    ({x_{I}, p_{\overline{I}}}, {x'_{I'}, p'_{\overline{I'}}}),
\end{equation*}
where $(x,p; x',p')$ are canonical coordinates on $T^*(M \times M) = T^*M \times T^*M$,
$I$, $I'$ are collections of indices (subsets of collection $\{1,\dots,n\}$) and
$\overline{I}$, $\overline{I'}$ are their complements.
Let $\mu$ be a fixed measure on $L$, and $\mu(w_I, w'_{I'})$ be its density written in local coordinates.
We assume that $\mu(w_I, w'_{I'})$ is a homogeneous function of degree $0$ with respect to
$(p_{\overline{I}}, p'_{\overline{I'}})$.

Given a Lagrangian submanifold $L$ and a measure $\mu$ on it,
under some special conditions imposed on pair $(L,\mu)$
(see~\cite{MSS1}),
one has a well-defined Maslov canonical operator $\mathcal{K}_{(L,\mu)}$
acting between suitable spaces of functions on $L$
and distributions on $M \times M$.
Given $\mathcal{K}_{(L,\mu)}$, one defines FIO $\Phi(L,a)$ on $M$
by the formula
\begin{equation}\label{defeq:FIOasIO}
    \Phi(L,a) u(x) = \int K_\Phi(x,x') u(x') \, d{x'},
\end{equation}
where $K_\Phi(x,x') = \mathcal{K}_{(L,\mu)} \, a(w_I, w'_{I'})$.
Explicitly  (in local coordinates),
\begin{multline}\label{defeq:FIOKernel}
        K_\Phi(x,x') = \mathcal{K}_{(L,\mu)} \, a(w_I, w'_{I'}) = \\
            = \mathcal{F}^{-1}_{p_{\overline{I}} \rightarrow x_{\overline{I}}}
              \mathcal{F}_{p'_{\overline{I'}} \rightarrow x'_{\overline{I'}}}
            \left\{
                e^{i S({x_{I}, p_{\overline{I}}}, {x'_{I'}, p'_{\overline{I'}}})} \,
                a_\mu({x_{I}, p_{\overline{I}}}, {x'_{I'}, p'_{\overline{I'}}})
            \right\} = \\
            = (2\pi)^{-(\lvert I \rvert + \lvert \overline{I'} \rvert)/2} \iint
                e^{i S({x_{I}, p_{\overline{I}}}, {x'_{I'}, p'_{\overline{I'}}})
                    + i p_{\overline{I}}x_{\overline{I}} - i p'_{\overline{I'}}x'_{\overline{I'}}} \,
                a_\mu({x_{I}, p_{\overline{I}}}, {x'_{I'}, p'_{\overline{I'}}})
                \, d{p_{\overline{I}}}\, d{p'_{\overline{I'}}}.
\end{multline}
Here the function $a_\mu$ equals
$$
    a_\mu(w_I, w'_{I'}) =
        \sqrt{\mu(w_I, w'_{I'})} \,
        a(w_I, w'_{I'}),
$$
and $S(w_I, w'_{I'})$ is the generating function of Lagrangian submanifold $L$.
The integrals~\eqref{defeq:FIOasIO} and~\eqref{defeq:FIOKernel} are considered as oscillatory integrals
(see, e.g,~\cite{Shu1}).

We assume that $a$ is a smooth function on $L$,
homogeneous of some degree $m$ with respect to $(p_{\overline{I}}, p'_{\overline{I'}})$.
We also assume that in local coordinates $a$ is smoothed at zero.

The function $S(w_I, w'_{I'})$ is defined as a phase function in a classical way.
By this we mean that $S(w_I, w'_{I'})$ is homogeneous of degree $1$ with respect to
$(p_{\overline{I}}, p'_{\overline{I'}})$, has no critical points within its support,
and the matrix of its second derivatives has the maximal rank (is non-degenerate).

A global FIO $\Phi(L,a)$ is glued up from the expressions of the form~\eqref{defeq:FIOasIO},
and is unique up to smoothing operators. If the functions $\mu(w_I, w'_{I'})$ and $a(w_I, w'_{I'})$
satisfy the above homogeneity conditions, then $\Phi(L,a)$ is a continuous operator in the spaces
$$
    \Phi(L,a): H^s(M) \longrightarrow H^{s-m}(M).
$$

Given a homogeneous canonical transformation $g: T^*_0M \rightarrow T^*_0M$, by the \textit{quantizing procedure}
we mean the mapping taking $g$ to FIO $\Phi(L,a)$ for which $L$ is the graph of $g$.
Namely,
$$
    L = \operatorname{graph} g = \{ (g(m'),m') \; | \; m' \in T^*_0M \} \subset T^*_0M \times T^*_0M.
$$
The corresponding FIO is called the \textit{quantized canonical transformation}.
Further on we denote it by $\Phi(g,a)$ instead of $\Phi(\operatorname{graph} g ,a)$.

\section{Main theorem}\label{sec:mainth}
In this section we state the main theorem of this paper.

First let us fix some notation. Let $M$ be closed smooth manifold of dimension $n$,
and $i: X \hookrightarrow M$ be a smooth embedding of a submanifold $X$.
Let
$$
    w = (x,t,p,\tau)
$$
be local coordinates on $T^*M$ such that $X$ in these coordinates is given
by the equation
$t = 0$, and coordinates $p$,$\tau$ are dual with respect to $x$,$t$
(we will use physical terminology and call $p$,$\tau$ the \textit{momentum coordinates}
and $x$,$t$ the \textit{physical coordinates}).
By $T^*_0M = T^*M \setminus \{0\}$ we denote, as above,
the cotangent bundle $T^*M$ with removed zero section.
By
$$
    (w,w') = (x,t,p,\tau; \, x',t',p',\tau')
$$
we denote the local coordinates on the product $T^*(M \times M) = T^*M \times T^*M$.

Let
\begin{equation}\label{eq:canonicalTransformationFunctions}
    g: (x',t',p',\tau') \longmapsto (x,t,p,\tau)
\end{equation}
be a canonical transformation written in the above local coordinates.
By $\operatorname{graph} g$ we denote the graph of $g$,
$$
    \operatorname{graph} g = \{ (w, w') \; | \; w = g(w') \} \subset T^*_0M \times T^*_0M.
$$
Evidently, $\operatorname{graph} g$ is a Lagrangian submanifold in
$T^*_0M \times T^*_0M$ with respect to the symplectic form $\pi_1^* \omega - \pi_2^* \omega$,
where $\omega = dp \wedge dx + d\tau \wedge dt$ is the standard symplectic form on $T^*M$,
and $\pi_1$, $\pi_2$ are projections on the corresponding factors in the product $T_0^*M \times T_0^*M$.

Given  an (arbitrary) submanifold $L$ in $T^*(M \times M)$,
we call the intersection
$$
    L|_X = L \cap T^*(M \times M)|_{X \times X}
$$
the \textit{restriction of $L$ associated with the embedding $X \hookrightarrow M$}.

Now we state the conditions of our theorem and make main assumptions.

Let $x_0 \in X \subset M$ be a fixed point with coordinates $(x,t) = (0,0)$.
In the sequel we will assume that $g$ satisfies the following conditions.
\begin{itemize}
    \item[A)] The following equalities of sets hold
        \begin{equation*}
            g((T^*_0M)_{x_0}) = (T^*_0M)_{x_0},
            \quad
            \pi_X(\operatorname{graph} g|_X) = \{x_0\} \times \{x_0\},
        \end{equation*}
        where by $(T^*_0M)_{x_0}$ we denote the fiber of $T^*_0M$ over $x_0$,
        and $\pi_X$ is the projection $T^*(M\times M)|_{X \times X} \rightarrow X \times X$.
    \item[B)]
        In $(T^*_0M)_{x_0}$ transformation $g$ is of the form
        \begin{equation*}
            g:(0,0,p',\tau') \mapsto (0,0,h(p',\tau'), \tau),
        \end{equation*}
        where $h(p',\tau')$ is a mapping homogeneous of degree $1$ with respect to $(p',\tau')$,
        which defines the following diffeomorphisms:
        \begin{equation}\label{eq:h(p',tau')determinatesDiffeomorphisms}
            \begin{array}{lll}
                h: \, \mathbb{R}^{n/2} \longrightarrow \mathbb{R}^{n/2}, &
                    \tau' \longmapsto p &
                    \text{for } \left\lvert p' \right\rvert \neq 0, \\
                h: \, \mathbb{R}^{n/2}\setminus\{0\} \longrightarrow \mathbb{R}^{n/2}\setminus\{0\}, &
                    \tau' \longmapsto p &
                    \text{for } \left\lvert p' \right\rvert = 0.
            \end{array}
        \end{equation}
\end{itemize}

Let us note that condition B) is satisfied only when $\dim X = \dim M /2 = n/2$.
It follows that $X$ is a submanifold of middle dimension, and all the coordinates
$x$,$t$,$p$,$\tau$ consist of $n/2$ variables.
Condition A) means that $g$ takes the fiber over point $x_0 \in X \subset M$ into itself,
and all the other fibers are being mapped to the complement of this fiber.

For simplicity we make some more assumptions.

Let $(x,t,p',\tau')$ be the local coordinates on $\operatorname{graph} g$, coming from a neighbourhood of
the fiber over $\{x_0\} \times \{x_0\}$.
It follows that the corresponding generating function $S$
in this neighbourhood has the form
$$
    S \equiv S(x,t,p',\tau').
$$
We assume that $S(x,t,p',\tau')$ is linear with respect to physical variables $x,t$, and
amplitude $a$ is independent of physical variables, that is, $a$ has the form
$$
    a \equiv a(p',\tau').
$$
Let $\mu$ be a fixed measure on $\operatorname{graph} g$,
and let its density 
be also independent of physical variables
(so it has the form $\mu \equiv \mu(p',\tau')$).
Finally, let $a(p',\tau')$ be homogeneous of degree $m$,
and $\mu(p',\tau')$ be homogeneous of degree $0$. The canonical transformation $g$
is assumed to be homogeneous of degree $1$ (therefore,
$S(x,t,p',\tau')$ is homogeneous of degree $1$ with respect to $(p',\tau')$).

Let us assume that the quantized canonical transformation $\Phi(g,a)$, built on the above data,
acts in the spaces $H^s(M) \rightarrow H^{s-m}(M)$ for $s < 0$ and $s - m - n/2 > 0$.

The following statement is the main result of the present paper.
\begin{theorem}\label{th:main}
    Under the above conditions, the reduced trace $i^!_0(\Phi(g,a))$,
    associated with the embedding $i: X \hookrightarrow M$,
    is a Fourier--Mellin operator
    localized at $x_0$.
\end{theorem}

\section{Proof of the main theorem}
We will divide the proof of Theorem~\ref{th:main} into two stages:
1) we show that $i^!(\Phi(g,a))$ is localized at $x_0$;
2) we describe an explicit structure of $i_0^!(\Phi(g,a))$ in local coordinates.

\subsection{Localization of $i^!(\Phi(g,a))$}
In this section we calculate the wave front set of $i^!(\Phi)$, in other words,
we study the problem of propagation of singularities for this operator.
It will turn out that its wave front set is concentrated in the fiber
over $x_0 \times x_0$, and this fact will imply the desired localization of our operator.

Within this section we do not suppose that assumptions discussed in Section~\ref{sec:mainth}
are made, but specify all the needed requirements explicitly.

As we have a new concept coming on stage, we start from fixing notations.
\begin{definition}
Let $A: C^\infty_0(M_2) \rightarrow \mathcal{D}'(M_1)$ be a linear operator between
function spaces (densities) on compact manifolds $M_2$ and $M_1$,
and let $K_A \in \mathcal{D}'(M_1 \times M_2)$ be its Schwartz kernel.
The \textit{twisted wave front set} (just \textit{wave front set} in the sequel) of the operator $A$
is a subset of $T^*_0(M_1 \times M_2)$, having the form
\begin{equation*}
    \operatorname{WF}'(A) =
    \{ (x_1,p_1; x_2,p_2) \in T^*_0M_1 \times T^*_0M_2 \; | \;
        (x_1,p_1; x_2,-p_2) \in \operatorname{WF}(K_A) \}.
\end{equation*}
Here by $(x_i,p_i)$ we denote the canonical coordinates
on $T^*_0M_i$, $i = 1,2$,
and $\operatorname{WF}(K_A)$ stands for the wave front set of the distribution $K_A$.
\end{definition}

Recall that if $A$ happens to be of the form $A \equiv \Phi(g,a)$, that is,
$A$ is the quantization of a canonical transformation $g$, then the following two relations hold:
\begin{multline*}
    \operatorname{WF}(\Phi(g,a) \, u) \subset \\ \subset
    \{(x_1,p_1) \in T^*_0M \; | \;
        \exists \, (x_2,p_2) \in \operatorname{WF}(u): (x_1,p_1; \, x_2,p_2) \in \operatorname{WF}'(\Phi(g,a)) \}
\end{multline*}
and
\begin{equation}\label{eq:WF'(FIO)subset(graphg)}
    \operatorname{WF}'(\Phi(g,a)) \subset \operatorname{graph} g
\end{equation}
(see~\S\S 25.1, 25.2 in~\cite{Hor4}).

Thus, $\operatorname{graph} g$ describes
propagation of singularities for the operator $\Phi(g,a)$.
Our nearest goal is to establish an analogous result for $i^!(\Phi(g,a))$.
It requires from us to introduce a new concept.

\begin{definition}\label{def:Sled(LagrangianManifold)}
Let $L$ be a submanifold in $T^*(M \times M)$.
The set
\begin{equation*}
    i^!(L) = \pi_{T^*X}(L|_X)
\end{equation*}
is called the \textit{trace
of submanifold $L$} associated with embedding $i: X \hookrightarrow M$.
Here $\pi_{T^*X}: T^*(M \times M)|_{X \times X} \rightarrow T^*(X \times X)$
is the natural projection corresponding to the embedding
$T(X \times X) \hookrightarrow T(M \times M)|_{X \times X}$ of vector bundles, and
$L|_X$ stands for the restriction of $L$ to $X$.
\end{definition}

Here is the desired result.
\begin{proposition}\label{prop:WF'(Sled)subsetSled(graphg)}
Let
\begin{equation}\label{eq:WF'(Sled)subsetSled(graphg)Condition}
    N_0^*X \cap g^{-1}(N_0^*X) = \emptyset,
\end{equation}
where $N_0^*X = N^*X \setminus \{0\}$ denotes the conormal bundle to $X$ in $T^*M$
with zero section deleted. Then
\begin{equation}\label{eq:WF'(SledFIO)subsetSled(graphg)}
    \operatorname{WF}'(i^!(\Phi(g,a))) \subset i^!(\operatorname{graph} g).
\end{equation}
\end{proposition}
\begin{remark}
Note that condition~\eqref{eq:WF'(Sled)subsetSled(graphg)Condition}
is satisfied whenever conditions A) and B) from Section~\ref{sec:mainth} are satisfied.
Thus, we will be able to apply Proposition~\ref{prop:WF'(Sled)subsetSled(graphg)}.
\end{remark}
\begin{proof}[Proof of Proposition~\ref{prop:WF'(Sled)subsetSled(graphg)}]
We will establish the formula~\eqref{eq:WF'(SledFIO)subsetSled(graphg)} in two steps:
1) we relate $i^!(\operatorname{graph} g)$ with a set obtained as a composition of three Lagrangian manifolds;
2) we compute $\operatorname{WF}'(i^!(\Phi(g,a)))$ using the fact that $i^!(\Phi(g,a))$ is a composition
of three integral operators.
It will turn out that these calculations result in the same set.

\textbf{Step 1.}
Recall that
$i^*$, $i_*$ are in fact Fourier integral operators
associated with the following Lagrangian manifolds
\begin{equation*}
    \begin{split}
        L_b & = \{ (x,t,p,\tau; \, x',p') \; | \; x' = x, \, p' = p, \, t = 0\} \subset (T^*M\times T^*X) \setminus \{0\}, \\
        L_c & = \{ (x,p; \, x',t',p',\tau') \; | \; x' = x, \, p' = p, \, t' = 0 \} \subset (T^*X\times T^*M) \setminus \{0\}
    \end{split}
\end{equation*}
(here the subscript ``b'' stands for ``boundary'' and the subscript ``c'' stands for ``coboundary'').

We are going to consider their compositions with $\operatorname{graph} g$.
Let us remind what this means.

\begin{definition}
Let $L_1$ be a subset in $T^*M_1\times T^*M_2$,
and let $L_2$ be a subset in $T^*M_2\times T^*M_3$.
Then their \textit{composition $L_2 \circ L_1$} is a subset in $T^*M_1\times T^*M_3$
defined as follows
\begin{equation*}
    L_2 \circ L_1 =
        \{ (w_1,w_3) \in T^*M_1\times T^*M_3 \; | \; \exists \, w_2\in T^*M_2: (w_1,w_2)\in L_1, \, (w_2,w_3)\in L_2 \}.
\end{equation*}
\end{definition}

\begin{lemma}\label{lemma:Sled(LagrangianManifold)AsComposition}
Let $L$ be a subset in $T^*(M \times M)$.
Then
$$
    i^!(L) = L_b \circ L \circ L_c.
$$
\end{lemma}
\begin{proof}
A straightforward calculation gives
\begin{equation*}
    L_b \circ L \circ L_c = \{(x,p;\,x',p') \; | \;
        \exists \, (\tau,\tau'): \,
        (x,0,p,\tau;\, x',0,p',\tau') \in L\}.
\end{equation*}
It is clear that the set on the right hand side
agrees with Definition~\ref{def:Sled(LagrangianManifold)}. Thus, the lemma is proved.
\end{proof}

\textbf{Step 2.}
Given an operator $A: C^\infty_0(M_2) \rightarrow \mathcal{D}'(M_1)$ and its Schwartz kernel $K_A$,
consider the following two sets
\begin{equation*}
    \begin{split}
        \operatorname{WF}'_{M_1}(A) & =
            \{ (x_1,p_1) \in T^*_0M_1 \; | \; \exists \, x_2 \in M_2: \, (x_1,p_1; \, x_2,0) \in \operatorname{WF}'(A) \}, \\
        \operatorname{WF}'_{M_2}(A) & =
            \{ (x_2,p_2) \in T^*_0M_2 \; | \; \exists \, x_1 \in M_1: \, (x_1,0; \, x_2,p_2) \in \operatorname{WF}'(A) \}.
    \end{split}
\end{equation*}

\begin{theorem}\label{th:WFComposition}
(see Theorem 8.2.14,~\cite{Hor1}).
Let $M_1$, $M_2$, $M_3$ be compact manifolds,
and let
\begin{equation*}
    \begin{split}
        & A: C^\infty_0(M_3) \longrightarrow \mathcal{D}'(M_2), \\
        & B: C^\infty_0(M_2) \longrightarrow \mathcal{D}'(M_1),
    \end{split}
\end{equation*}
be operators for which
\begin{equation}\label{eq:thWFCompositionCondition}
    \operatorname{WF}'_{M_2}(A) \cap \operatorname{WF}'_{M_2}(B) = \emptyset.
\end{equation}
Then the composition $B \circ A: C^\infty_0(M_3) \rightarrow \mathcal{D}'(M_1)$
is well-defined, and the following relation holds
\begin{equation}\label{eq:WFCompositionFormula}
    \begin{split}
        \operatorname{WF}'(B \circ A) \subset
        \operatorname{WF}'(B) \circ WF'(A) \,
        & \bigcup \, \{(x_1,0) \in T^*M_1\} \times \operatorname{WF}'_{M_3}(A) \, \bigcup \\
        & \bigcup \, \operatorname{WF}'_{M_1}(B) \times \{(x_3,0) \in T^*M_3\}.
    \end{split}
\end{equation}
\end{theorem}

We are going to apply theorem~\ref{th:WFComposition} twice: first to the composition
$\Phi i_*$ and then to the composition $i^* (\Phi i_*)$. Having this in mind, we note
that the following obvious inclusions are valid
\begin{equation}\label{eq:WFof(co)boundaryOperators}
    \operatorname{WF}'(i^*) \subset L_b, \quad \operatorname{WF}'(i_*) \subset L_c.
\end{equation}

Now, consider the composition
$$
    \Phi i_*: C^\infty_0(X) \longrightarrow \mathcal{D}'(M).
$$
We start from checking condition~\eqref{eq:thWFCompositionCondition} for bigger sets.
We have
\begin{equation*}
    \begin{split}
        \operatorname{WF}'_{M}(\Phi)
            & \subset \{(x,t,p,\tau) \in T^*_0M \; | \;
                \exists \, (x',t'): \, (x,t,p,\tau; \, x',t',0,0) \in \operatorname{graph} g \} = \\
            & = \{(x,t,p,\tau) \in T^*_0M \; | \; \exists \, (x',t'): \, (x',t',0,0) = g^{-1}(x,t,p,\tau) \}. 
    \end{split}
\end{equation*}
But the last set is empty, since any point in $T^*M$ with coordinates $(x',t',0,0)$
belongs to the zero section of $T^*M$, so it can not happen to be in the image of $g$.
Thus, condition~\eqref{eq:thWFCompositionCondition} is satisfied,
and we have the following
\begin{equation}\label{eq:WFCompositionCob}
    \begin{split}
        \operatorname{WF}'(\Phi i_*) \, \subset \, \operatorname{graph} g \, \circ \, L_c \,
            & \bigcup \, \{(x,0)\} \times \operatorname{WF}'_M(\Phi) \, \bigcup \\
            & \bigcup \, \operatorname{WF}'_X(i_*) \times \{ (x',t',0,0) \},
    \end{split}
\end{equation}
where $\operatorname{WF}'_M(\Phi)$ corresponds to the second factor in $T^*M \times T^*M$.

Consider the last term in union~\eqref{eq:WFCompositionCob}. We have
\begin{equation*}
    \begin{split}
        \operatorname{WF}'_{X}(i_*)
            & \subset \{(x,p) \in T^*_0M \; | \; \exists \, (x',t') \in M: \, (x,p; x',t',0,0) \in L_c \} = \\
            & = \{(x,p) \in T^*_0M \; | \; x' = x, \, t' = 0, \, p = 0 \} = \\
            & = \{(x,0) \in T^*_0M \} = \emptyset,
    \end{split}
\end{equation*}
since $T^*_0M$ does not contain the zero section. Thus the last term in union~\eqref{eq:WFCompositionCob} is empty.

Consider the second term in~\eqref{eq:WFCompositionCob}. By definition,
\begin{equation*}
    \operatorname{WF}'_{M}(\Phi)
    \subset \{(x',t',p',\tau') \in T^*_0M \; | \; \exists \, (x,t) \in M: \, (x,t,0,0) = g(x',t',p',\tau')\}.
\end{equation*}
Again, since $g$
is a diffeomorphism $T^*_0M \rightarrow T^*_0M$, it follows that $\operatorname{WF}'_{M}(\Phi) = \emptyset$.
Hence the second term in~\eqref{eq:WFCompositionCob} is also empty.

Eventually we have
$$
    \operatorname{WF}'(\Phi i_*) \subset L \circ L_c.
$$

Now, consider the composition $i^* (\Phi i_*)$. Arguing as above, we see that
in general condition~\eqref{eq:thWFCompositionCondition} is not satisfied.
Namely, having calculated $\operatorname{WF}'_M(i^*)$ and $\operatorname{WF}'_M(\Phi i_*)$
we get
\begin{equation*}
    \operatorname{WF}'_M(i^*)\cap \operatorname{WF}'_M(\Phi i_*) \subset N_0^*X \cap g^{-1}(N_0^*X).
\end{equation*}
However, in our setting the set on the right hand side is empty by assumption,
and this allows us to apply Theorem~\ref{th:WFComposition} in this situation as well.
A calculation shows that
$$
    \operatorname{WF}'(i^* \circ (\Phi \circ i_*)) \subset L_b \circ (L \circ L_c).
$$
This completes step 2.

Applying Lemma~\ref{lemma:Sled(LagrangianManifold)AsComposition},
we get the desired result.
Proposition~\ref{prop:WF'(Sled)subsetSled(graphg)} is proved.
\end{proof}

Now we return to the setting of theorem~\ref{th:main}.
\begin{lemma}\label{lemma:localizedInAPoint}
    $i^!(\Phi(g,a))$ is localized at $x_0$.
\end{lemma}
\begin{proof}
It is easy to see that in terms of Definition~\ref{def:Sled(LagrangianManifold)},
condition A) is equivalent to the following equality of sets
\begin{equation}\label{eq:sled(graphg)equalsFiber}
    i^!(\operatorname{graph} g) = \bigl(T^*_0(X \times X)\bigr)_{(x_0,x_0)},
\end{equation}
where by $(T^*_0(X \times X))_{(x_0,x_0)}$ we denote the fiber of $T^*_0(X \times X)$ over $(x_0,x_0)$.
Now, the desired result immediately follows from Proposition~\ref{prop:WF'(Sled)subsetSled(graphg)}.

Indeed, consider operators $\varphi \, i^!(\Phi)$ and $i^!(\Phi) \, \varphi$,
where by $\varphi$ we mean a multiplication operator by a smooth function
$\varphi \in C^\infty(X)$ vanishing in a neighbourhood of $x_0$.
The kernels of these operators have the form
$$
    K_{\varphi \, i^!(\Phi)}(x,x') = \varphi(x) \, K_{i^!(\Phi)}(x,x'), \quad
    K_{i^!(\Phi) \, \varphi}(x,x') = \varphi(x') \, K_{i^!(\Phi)}(x,x').
$$
Since multiplication by a smooth function does not increase the wave front set, it follows that
$\operatorname{WF}'(\varphi \, i^!(\Phi))$ and $\operatorname{WF}'(i^!(\Phi) \, \varphi)$
are subsets in $\operatorname{WF}'(i^!(\Phi)) \subset i^!(\operatorname{graph} g)$,
and, since~\eqref{eq:sled(graphg)equalsFiber} holds,
they must be subsets in the fiber $(T^*_0(X \times X))_{(x_0,x_0)}$.
But this is impossible because $x_0$ lies outside the support of $\varphi$.
Therefore, both $\operatorname{WF}'(\varphi \, i^!(\Phi))$ and $\operatorname{WF}'(i^!(\Phi) \, \varphi)$
are empty, hence, $\varphi \, i^!(\Phi)$ and $i^!(\Phi) \, \varphi$
turn out to be integral operators with smooth kernels. It follows that they are compact,
and this completes the proof.
\end{proof}

\subsection{Calculations in local coordinates}

Since the trace $i^!(\Phi(g,a))$ is localized at $x_0$,
it suffices to consider it in a fixed coordinate neighbourhood of $x_0$.
Having this in mind we can assume that $M$ is realised as
(a domain in) Euclidean space $\mathbb{R}^n$, and $X$ is its subspace $\mathbb{R}^{n-\nu} \subset \mathbb{R}^n$
(recall that by our assumptions $\nu = n/2$).

From now on we write $\mathbb{R}^n$ or $\mathbb{R}^n_{x,t}$
instead of $M$, and $\mathbb{R}^{n/2}$ or $\mathbb{R}^{n/2}_x$ instead of $X$.
The embedding $i: X \hookrightarrow M$ is considered as embedding
$\mathbb{R}^{n/2}_x \hookrightarrow \mathbb{R}^{n}_{x,t}$ of Euclidean subspaces.
By $K_A$ we denote the kernel of an integral operator $A$.
Every integral in this section is being interpreted in the sense of distributions
(as an oscillatory integral)
unless otherwise stated or is clear from context.

Now, our operator $\Phi = \Phi(g,a)$ can be rewritten in the form
\begin{equation}\label{defeq:localFIO}
    \Phi u(x,t) =
        \iint
        \mathcal{F}_{(p',\tau') \rightarrow (x',t')} \, \biggl\{ e^{i S(x,t,p',\tau')} a_\mu(p',\tau')\biggr\}
        \, u(x',t') \, d{x'}\, d{t'}.
\end{equation}
In other words, $\Phi$ is an integral operator with the following Schwartz kernel
\begin{equation*}
    K_\Phi(x,t,x',t') = \mathcal{F}_{(p',\tau') \rightarrow (x',t')} \, \biggl\{ e^{i S(x,t,p',\tau')} a_\mu(p',\tau')\biggr\}.
\end{equation*}

\begin{proposition}\label{prop:sledFIO}
$i^!(\Phi)$ is an integral operator with the following Schwartz kernel
\begin{equation}\label{eq:kernelSledFIO}
    K_{i^!(\Phi)}(x,x') =
        \int \mathcal{F}_{p' \rightarrow x'} \,
        \biggl\{ e^{i S(x,0,p',\tau')} a_\mu(p',\tau') \biggr\} \, d{\tau'}.
\end{equation}

\end{proposition}
\begin{remark}
It is easy to see that the desired result can be quickly obtained from
the direct substituting the formulas for $i^*$, $i_*$
(we mean~\eqref{defeq:defBOp} and~\eqref{defeq:defCobOp}) into the composition
$i^* \Phi i_*$, where $\Phi$ is defined by~\eqref{defeq:localFIO}.
However, we will choose another option and provide the calculation
in dual (with respect to the Fourier transform) coordinates.
This is reasonable, since we expect to get the structure of a Fourier--Mellin operator.
\end{remark}

\begin{proof}[Proof of Proposition~\ref{prop:sledFIO}]
It is easy to see that the operators $i^*$, $i_*$ in dual coordinates have the forms
\begin{equation*}
    \begin{split}
        & \widetilde{i^*} =
            \mathcal{F}_{x \rightarrow p} \, i^* \, \mathcal{F}^{-1}_{(p,\tau) \rightarrow (x,t)} \widetilde{u}(p,\tau) =
            \int \widetilde{u}(p,\tau) \, d{\tau}, \\
        & \widetilde{i_*} =
            \mathcal{F}_{(x,t) \rightarrow (p,\tau)} \, i_* \, \mathcal{F}^{-1}_{p \rightarrow x} \widetilde{u}(p) =
            \widetilde{u}(p) \otimes 1_\tau
    \end{split}
\end{equation*}
(here $1_\tau$ stands for the function, ``depending'' on $\tau$ variable, which is identically equal to $1$).
In other words, $\widetilde{i^*}$ acts by integrating with respect to the fiber variable
and $\widetilde{i_*}$ is an extension by a constant.

Having obtained $\widetilde{i^*}$ and $\widetilde{i_*}$, we need an
analogous result for $\Phi(g,a)$, that is, a method on how to convert it into dual coordinates.
We start from some general considerations.

Within the body of the next lemma we forget the submanifold structure
(assume for a while that $x$ are complete coordinates on $\mathbb{R}^n$).
\begin{lemma}\label{lemma:dualSchwartzKernelOperator}
Let $A$ be an integral operator in the spaces
$$
    A: S(\mathbb{R}^n) \longrightarrow S'(\mathbb{R}^n),
$$
where $S(\mathbb{R}^n)$ is Schwartz space of rapidly decreasing functions.
Let $K_A(x,x') \in S'(\mathbb{R}^n_x \times \mathbb{R}^n_{x'})$ be the Schwartz kernel of $A$.
Then the operator $\widetilde{A} = {\mathcal{F}_{x \rightarrow p}} \, A \, {\mathcal{F}^{-1}_{p \rightarrow x}}$
is an integral operator in the same spaces, and its Schwartz kernel $K_{\widetilde{A}}(p,p')$ has the form
\begin{equation*}
    K_{\widetilde{A}}(p,p')
        = \mathcal{F}_{x \rightarrow p} \mathcal{F}^{-1}_{x' \rightarrow p'} \, K_{A}(x,x').
\end{equation*}
\end{lemma}
\begin{proof}
Given a test function $\widetilde{\varphi}(p) \in S(\mathbb{R}^n)$, we have
\begin{equation*}
    \begin{split}
    \langle \widetilde{A} \widetilde{u}(p), \widetilde{\varphi}(p) \rangle &
        = \langle \mathcal{F}_{x \rightarrow p} \, A \, \mathcal{F}^{-1}_{p \rightarrow x} \widetilde{u}(p), \, \widetilde{\varphi}(p) \rangle = \\ &
        = \langle A \, \mathcal{F}^{-1}_{p \rightarrow x} \, \widetilde{u}(p), \, \mathcal{F}_{p \rightarrow x} \, \widetilde{\varphi}(p) \rangle = \\ &
        = \langle K_A(x,x'), \, {\mathcal{F}^{-1}_{p' \rightarrow x'}} \widetilde{u}(p') \otimes \mathcal{F}_{p \rightarrow x} \widetilde{\varphi}(p) \rangle = \\ &
        = \langle K_A(x,x'), \, {\mathcal{F}^{-1}_{p' \rightarrow x'}} \mathcal{F}_{p \rightarrow x} \,
            \left[ \widetilde{u}(p') \otimes  \widetilde{\varphi}(p) \right] \rangle = \\ &
        = \langle \mathcal{F}_{x \rightarrow p} \mathcal{F}^{-1}_{x' \rightarrow p'} \, K_A(x,x'), \, \widetilde{u}(p') \otimes \widetilde{\varphi}(p) \rangle.
    \end{split}
\end{equation*}
This completes the proof.
\end{proof}

Now we are back to coordinates $(x,t)$.
Let $A$ be an integral operator on $\mathbb{R}^n_{x,t}$ and let $K_A(x,t,x',t')$ be its Schwartz kernel.
We are interested in $\widetilde{i^!(A)}$ by which we denote the trace of $A$ in dual coordinates.

First, at least formally, we have
\begin{equation*}
    \begin{split}
        \widetilde{i^!(A)}\widetilde{u}(p) &
        = \mathcal{F}_{x \rightarrow p} \, i^* A \, i_* \, \mathcal{F}^{-1}_{p \rightarrow x} \, \widetilde{u}(p) = \\ &
        = \mathcal{F}_{x \rightarrow p} \, i^* A \, \mathcal{F}^{-1}_{(p,\tau) \rightarrow (x,t)} \,
            \bigl[ \widetilde{u}(p) \otimes 1_\tau\bigr] = \\ &
        = \mathcal{F}_{x \rightarrow p} \, i^* \mathcal{F}^{-1}_{(p,\tau) \rightarrow (x,t)} \, \widetilde{A} \,
            \bigl[ \widetilde{u}(p) \otimes 1_\tau\bigr] = \\ &
        = \mathcal{F}_{x \rightarrow p} \, i^* \mathcal{F}^{-1}_{(p,\tau) \rightarrow (x,t)} \,
            \iint K_{\widetilde{A}}(p,\tau,p',\tau') \, \widetilde{u}(p') \, d{p'}\, d{\tau'} =  \\ &
        = \iiint K_{\widetilde{A}}(p,\tau,p',\tau') \, \widetilde{u}(p') \, d{p'} \, d{\tau}\, d{\tau'}.
    \end{split}
\end{equation*}
Thus, the kernel of $\widetilde{i^!(A)}$ can be expressed via the kernel of $A$
by integrating the latter with respect to conormal variables. Namely,
\begin{equation}\label{eq:kernelDualSled}
    K_{\widetilde{i^!(A)}}(p,p') = \iint K_{\widetilde{A}}(p,\tau,p',\tau') \, d{\tau}\, d{\tau'}.
\end{equation}

We claim that integral~\eqref{eq:kernelDualSled} makes sense and is well-defined
if one considers $\Phi$ defined by~\eqref{defeq:localFIO} instead of an arbitrary operator $A$.
In other words, the following result takes place
\begin{lemma}\label{lemma:dualSledKernel}
In dual coordinates $i^!(\Phi)$ becomes an integral operator
with Schwartz kernel having the form
\begin{equation}\label{eq:kernelDualSledFIO}
    K_{\widetilde{i^!(\Phi)}}(p,p') =
        \int \mathcal{F}_{x \rightarrow p} \, \biggl\{ e^{i S(x,0,p',\tau')} a_\mu(p',\tau')\biggr\} \, d{\tau'}.
\end{equation}
\end{lemma}
\begin{proof}
By Lemma~\ref{lemma:dualSchwartzKernelOperator}, we have
\begin{equation*}
    \begin{split}
        K_{\widetilde{\Phi}}(p,\tau,p',\tau') &
            = \mathcal{F}_{(x,t) \rightarrow (p,\tau)} \mathcal{F}^{-1}_{(x',t') \rightarrow (p',\tau')} K_\Phi(x,t,x',t') = \\ &
            = \mathcal{F}_{(x,t) \rightarrow (p,\tau)} \, \biggl\{ e^{i S(x,t,p',\tau')} a_\mu(p',\tau') \biggr\}.
    \end{split}
\end{equation*}
After substituting this expression into~\eqref{eq:kernelDualSled} we get
$$
    K_{\widetilde{i^!(\Phi)}}(p,p') =\iint \mathcal{F}_{(x,t) \rightarrow (p,\tau)} \,
    \biggl\{ e^{i S(x,t,p',\tau')} a_\mu(p',\tau') \biggr\} \, d{\tau}\, d{\tau'}.
$$
The obtained formula is well-defined in the sense of oscillatory integrals.
Indeed, the corresponding phase function has the form
$$
    (p,p'; x,t,\tau,\tau') \longmapsto -px -\tau t + S(x,t,p',\tau')
$$
and possess all the required properties (see Section~\ref{subsec:operatorFIO});
we note that the homogeneity of degree $1$ with respect to variables $(x,t,\tau,\tau')$
is provided due to the condition that $S$ is linear in physical variables.

It remains to note that the composition of the Fourier transform
$\mathcal{F}_{t \rightarrow \tau}$ and subsequent integration with respect to $\tau$
can be replaced with the substitution $t = 0$. This completes the proof.
\end{proof}

Now, we can derive formula~\eqref{eq:kernelSledFIO} from Lemma~\ref{lemma:dualSledKernel}
and Lemma~\ref{lemma:dualSchwartzKernelOperator}. Thus, Proposition~\ref{prop:sledFIO} is proved.
\end{proof}

\subsection{Calculations in the dual space}
We continue considering $i^!(\Phi)$ locally in dual coordinates.
First we pay attention to the generating function $S$.

\begin{lemma}\label{lemma:S(x,0,p',tau')equalsxh(p',tau')}
Let $S$ be the generating function of $\operatorname{graph} g$ viewed as a Lagrangian submanifold.
Then in a neighbourhood of the fiber over $\{x_0\} \times \{x_0\}$
the restriction $S(x,0,p',\tau')$ of $S(x,t,p',\tau')$ has the following form
\begin{equation}\label{eq:linearSonX}
    S(x,0,p',\tau') = x \, h(p',\tau'),
\end{equation}
where $h$ is the same as in condition B) (see~\eqref{eq:h(p',tau')determinatesDiffeomorphisms}).
\end{lemma}
\begin{proof}
The fact that $S$ is the generating function
means that $\operatorname{graph} g$ viewed as a submanifold of $T^*_0(M \times M)$
is determined by the following equations
\begin{equation}\label{eqs:lagrangianCoordsFromS}
    x' = \frac{\partial S}{\partial p'}, \quad t' = \frac{\partial S}{\partial \tau'}, \quad
    p = \frac{\partial S}{\partial x}, \quad \, \tau = \frac{\partial S}{\partial t}.
\end{equation}

On the other hand, by virtue of the condition that $S$ is linear in physical variables we have
\begin{equation}\label{eq:linearS}
    S(x,t,p',\tau') = S_0(t,p',\tau') + x \frac{\partial S}{\partial x}(t,p',\tau').
\end{equation}

Note that $S_0(t,p',\tau')$ is necessarily zero.
Indeed, because of condition A) and by virtue of~\eqref{eqs:lagrangianCoordsFromS}, we have
$$
    \frac{\partial S}{\partial p'}(0,0,p',\tau') = 0, \quad
    \frac{\partial S}{\partial \tau'}(0,0,p',\tau') = 0.
$$
From~\eqref{eq:linearS} we get
$$
    \frac{\partial S_0}{\partial p}(0,p',\tau') = 0, \quad
    \frac{\partial S_0}{\partial \tau}(0,p',\tau') = 0,
$$
and it follows that $S_0(0,p',\tau')$ is a constant;
since $S(x,t,p',\tau')$ is homogeneous in momentum variables,
this constant is zero.

Now, by setting
\begin{equation}\label{eq:defhmap}
    h(p',\tau') = \frac{\partial S}{\partial x}(0,p',\tau'),
\end{equation}
we get~\eqref{eq:linearSonX}.
From~\eqref{eqs:lagrangianCoordsFromS} and~\eqref{eq:defhmap} it follows
that this new $h$ is nothing but the $p$-component of $g$;
so it is essentially the same $h$ from condition B).
The lemma is proved.
\end{proof}

Now, we are ready to study $\widetilde{i^!(\Phi)}$, that is, the operator $i^!(\Phi)$ transferred to dual coordinates.
\begin{proposition}\label{prop:dualOurSledIsIntegralOperator}
The operator $\widetilde{i^!(\Phi)}$ is of the form
\begin{equation}\label{eq:dualOurSledAsIntegralOperatorSecond}
    \widetilde{u}(p) \longmapsto \int \psi(p) \, b^0(p,p') \, \psi'(p') \, \widetilde{u}(p') \, d{p'},
\end{equation}
where $b^0(p,p')$ is a homogeneous function of degree $m$,
$\psi(p)$ are $\psi'(p')$ are cut-off functions which equal to $0$ in a neighbourhood of zero and $1$ at infinity.
The operator~\eqref{eq:dualOurSledAsIntegralOperatorSecond} acts continuously in the spaces
$$
    \widetilde{H}^s(\mathbb{R}^{n/2}_p) \longrightarrow \widetilde{H}^{s-m-n/2}(\mathbb{R}^{n/2}_p),
$$
(here $\widetilde{H}^s = \mathcal{F}_{x \rightarrow p}H^s$ is dual to the Sobolev space $H^s$ with respect to the Fourier transform),
and is independent of the choice
of functions $\psi$ and $\psi'$ up to operators which are compact in physical coordinates.
\end{proposition}

\begin{proof}
\textbf{Step 1.}
\textit{Reduction to an integral operator with a homogeneous kernel.}
Consider the kernel $K_{\widetilde{i^!(\Phi)}}$ of $\widetilde{i^!(\Phi)}$.
Using~\eqref{eq:linearSonX}, we get from~\eqref{eq:kernelDualSledFIO}
\begin{equation}\label{eq:kernelDualSledOurFIO}
    K_{\widetilde{i^!(\Phi)}}(p,p') =
        \int \mathcal{F}_{x \rightarrow p}
            \left\{ e^{i x h(p',\tau')} a_\mu(p',\tau')\right\}
        \, d{\tau'}.
\end{equation}

We want to make the change of variables $\tau' \mapsto \eta$ by substituting $\eta = h(p',\tau')$.
To justify this action, let us note that the mapping
\begin{equation}\label{eq:mapForSubstitution}
    \mathbb{R}^{n}\setminus\{0\} \longrightarrow \mathbb{R}^{n}\setminus\{0\}, \quad (p',\tau') \longmapsto (p',h(p',\tau'))
\end{equation}
is well-defined and is a homogeneous diffeomorphism of degree $1$. Indeed, this easily follows from condition B).

The inverse mapping can be written in the form
\begin{equation}\label{eq:inverseMapForSubstitution}
    \mathbb{R}^{n}\setminus\{0\} \longrightarrow \mathbb{R}^{n}\setminus\{0\}, \quad (p',\eta) \longmapsto (p',h^{-1}(p',\eta)),
\end{equation}
where $h^{-1}$ is the inverse mapping of $h$ for fixed $p'$.
Since~\eqref{eq:inverseMapForSubstitution} is a diffeomorphism,
we see that $h^{-1}(p',\eta)$ is a smooth mapping outside zero,
and so is its Jacobian which we denote by $J(p',\eta)$.

Now, after the change of variables $\tau' \mapsto \eta = h(p',\tau')$ in~\eqref{eq:kernelDualSledOurFIO},
we get
\begin{equation*}
    K_{\widetilde{i^!(\Phi)}}(p,p') =
        \int
            \mathcal{F}_{x \rightarrow p} \left\{e^{i x \eta}\right\} \,
            {\left\lvert J(p',\eta) \right\rvert} \,
            a_\mu(p',h^{-1}(p',\eta))
        \, d{\eta}.
\end{equation*}
Note that $\mathcal{F}_{x \rightarrow p} \left\{e^{i x \eta}\right\}$ yields a Dirac delta-function
depending on $p$ variable,
and further integration with respect to $p$ is simply application of this delta-function.
So we have
\begin{equation*}
    \begin{split}
        K_{\widetilde{i^!(\Phi)}}(p,p') & =
        \int
            \delta(p-\eta) \,
            {\left\lvert J(p',\eta) \right\rvert} \,
            a_\mu(p',h^{-1}(p',\eta))
        \, d{\eta} = \\ & =
        {\left\lvert J(p',p) \right\rvert} \, a_\mu(p',h^{-1}(p',p)).
    \end{split}
\end{equation*}

Now, by setting
\begin{equation*}
    b(p,p') = {\left\lvert J(p',p) \right\rvert} \, a_\mu(p',h^{-1}(p',p)),
\end{equation*}
we obtain that $\widetilde{i^!(\Phi)}$ is an operator of the form
\begin{equation}\label{eq:dualOurSledAsIntegralOperatorFirst}
    \widetilde{i^!(\Phi)} \widetilde{u}(p) = \int b(p,p') \widetilde{u}(p') \, d{p'}.
\end{equation}

We claim that the function $b(p,p')$ is homogeneous of degree $m$ and smooth outside zero.
Indeed, the homogeneity is by the construction
and the smoothness follows from the smoothness of~\eqref{eq:inverseMapForSubstitution}.

\textbf{Step 2.} \textit{Smoothing the kernel by introducing cut-off functions}. 
Let $b^0(p,p')$ be a homogeneous function of degree $m$ coinciding with $b(p,p')$ at infinity.
We want to show that the kernel $b(p,p')$ of operator~\eqref{eq:dualOurSledAsIntegralOperatorFirst}
can be replaced by a function of the form $\psi(p) b^0(p,p') \psi'(p')$, where
$\psi$ and $\psi'$ are smooth functions that cut out the origin,
up to operators
which are compact in physical coordinates.
Let us estimate the norm of an operator with the kernel $b - \psi b^0 \psi'$.

Let $\psi$ and $\psi'$ be smooth functions which equal $0$ in a neighbourhood of zero and
equal to $1$ at infinity. Consider an operator ${R}$ acting as follows
\begin{equation*}
    {R} \widetilde{u}(p) = \int \left[b(p,p') - \psi(p) \, b^0(p,p') \, \psi'(p')\right] \widetilde{u}(p') \, d{p'}.
\end{equation*}

\begin{lemma}\label{lemma:differenceNormEstimate}
Let $s <0$ and $s-m-n/2 > 0$.
Then ${R}$ acts continuously in the spaces
\begin{equation}\label{eq:operator{R}}
    {R}: \widetilde{H}^s(\mathbb{R}^{n/2}) \longrightarrow \widetilde{H}^{s-m-n/2+\varepsilon}(\mathbb{R}^{n/2})
\end{equation}
for any $\varepsilon$ satisfying $0 \leq \varepsilon < n/4$.
\end{lemma}
\begin{proof}

\textbf{Step 1.} Note that we can smooth the kernel of operator~\eqref{eq:dualOurSledAsIntegralOperatorFirst}
by multiplying it by a cut-off function of two variables. More preciously, consider an operator
\begin{equation}\label{eq:dualOurSledAsIntegralOperatorFirstCorrected}
    \widetilde{u}(p) \longmapsto \int \chi(p,p') \, b^0(p,p') \, \widetilde{u}(p') \, d{p'},
\end{equation}
where $\chi(p,p')$ is a smooth function which equals $0$ in a neighbourhood of zero and
equals $1$ at infinity. The difference of operators~\eqref{eq:dualOurSledAsIntegralOperatorFirst}
and~\eqref{eq:dualOurSledAsIntegralOperatorFirstCorrected}
is an operator with the kernel $\chi b - b^0 = (1-\chi) b^0$ which is a function vanishing at infinity.
Such an operator is smoothing in physical coordinates,
so we can consider~\eqref{eq:dualOurSledAsIntegralOperatorFirstCorrected}
instead of~\eqref{eq:dualOurSledAsIntegralOperatorFirst}.

Using this argument,
let us think that ${R}$ is actually induced by kernel $(\chi - \psi \psi')b^0$.

\textbf{Step 2.}
Fix cut-off functions $\chi$ and $\psi$, $\psi'$ as follows:
let $\chi(p,p') = 1$ if ${\lvert p \rvert}, {\lvert p' \rvert} \geq C/2$,
let $\psi(p) = 1$ if ${\lvert p \rvert} \geq C$,
and let $\psi(p') = 1$ if ${\lvert p' \rvert} \geq C$; here $C$ is some positive constant.

We are going to estimate the following expression
\begin{equation*}
    \left\| {R}\widetilde{u} \right\|^2_{\widetilde{H}^{s'}} =
        \int
        \left( 1 + {\lvert p \rvert}^2 \right)^{s'}
        {\left\lvert {\,
            \int b^0(p,p')\left[\chi(p,p')-\psi(p) \, \psi'(p')\right] \, \widetilde{u}(p') \, d{p'}
        \,} \right\rvert}^2
        \, d{p}
\end{equation*}
for some $s' > 0$ (which we specify later).

Applying Cauchy--Schwarz inequality to the inner integral
and using the assumption that $s < 0$ and $s'>0$, we obtain
\begin{equation*}
    \begin{split}
        \left\| {R}\widetilde{u} \right\|^2_{\widetilde{H}^{s'}}
            & \leq \iint \left( 1 + {\lvert p \rvert}^2 + {\lvert p' \rvert}^2 \right)^{s'-s}
                {\left\lvert {b^0(p,p') \left[\chi(p,p')-\psi(p) \, \psi'(p')\right]} \right\rvert}^2 \,
                \left\| \widetilde{u} \right\|^2_{\widetilde{H}^s}
                \, d{p'} \, d{p} = \\
            & = \left\| \widetilde{u} \right\|^2_{\widetilde{H}^s}
                \iint \left( 1 + {\lvert p \rvert}^2 + {\lvert p' \rvert}^2 \right)^{s'-s}
                {\bigl\lvert b^0(p,p') \bigr\rvert}^2{\bigl\lvert \chi(p,p')-\psi(p) \, \psi'(p') \bigr\rvert}^2
                \, d{p'} \, d{p}.
    \end{split}
\end{equation*}
The difficulty here is due to the fact that $\chi - \psi\psi'$ does not vanish at infinity,
so we can not immediately conclude that the entire integral is bounded.
However, this can be easily resolved as follows
(note that the trouble occurs only when $p = 0$ or $p' =0$).

We represent the domain of integration $\mathbb{R}^{n}_{p,p'}$ as a union
$$
    \mathbb{R}^{n}_{p,p'} = W \cup U \cup U', \quad
$$
where
$W = \{{\lvert p \rvert} \geq C/2 \text{ and } {\lvert p' \rvert} \geq C/2\}$,
$U = \{{\lvert p' \rvert} \leq C\}$,
and $U' = \{{\lvert p \rvert} \leq C\}$; here $C$ is some positive constant (not necessarily the same as one above).
Let us estimate three integrals corresponding to these new domains.

1) The integral over $W$ is bounded for any $s'$,
\begin{equation*}
        \int_W \left(1 + {\lvert p \rvert}^2 + {\lvert p' \rvert}^2 \right)^{s'-s}
                {\bigl\lvert b^0(p,p') \bigr\rvert}^2
                {\bigl\lvert \chi(p,p')-\psi(p) \psi'(p') \bigr\rvert}^2
                \, d{p'} \, d{p} < \infty \quad \forall s',
\end{equation*}
since $\chi-\psi \psi'$ vanishes at infinity in $W$.

2) Over $U$ we have
\begin{equation*}
    \begin{split}
        \int_{U} \big( 1 + & {\lvert p \rvert}^2 + {\lvert p' \rvert}^2 \big)^{s'-s}
                {\bigl\lvert b^0(p,p') \bigr\rvert}^2
                {\bigl\lvert \chi(p,p')-\psi(p) \psi'(p') \bigr\rvert}^2
                \, d{p'} \, d{p} \leq \\ &
        \leq \int_{U} \big( 1 + {\lvert p \rvert}^2 + {\lvert p' \rvert}^2 \big)^{s'-s}
                {\bigl\lvert b^0(p,p') \bigr\rvert}^2
                \, d{p'} \, d{p} \leq \\ &
        \leq C_1 \int_{U}
                \big(1+ {\lvert p \rvert}^2+{\lvert p' \rvert}^2\big)^{s'-s+m}
                \, d{p'} \, d{p} < \\ &
        < C_2 \int_{\mathbb{R}^{n/2}}
                \big(1+{\lvert p \rvert}^2\big)^{s'-s+m}
                \, d{p}
    \end{split}
\end{equation*}
(here we use the homogeneity of $b^0$ and the obvious fact that $\chi - \psi \psi' \leq 1$).
The latter integral converges for $s'-s+m < -n/4$. Set $s' = s-m-n/2+\varepsilon > 0$
(recall that $s-m-n/2>0$ by assumption).
Then the integral converges for $0 \leq \varepsilon < n/4$.

3) The integral over $U'$ is completely similar to the integral over $U$, and
the same relations provide that it is convergent.

The obtained estimate shows that operator~\eqref{eq:operator{R}} is
bounded whenever $0 \leq \varepsilon < n/4$.
Lemma~\ref{lemma:differenceNormEstimate} is proved.
\end{proof}

Now, we recall that  the relations $s < 0$, $s - m - n/2 > 0$
are satisfied by assumption,
so lemma~\eqref{lemma:differenceNormEstimate} ensures that ${R}$
transferred to physical coordinates
is a smoothing operator in the spaces $H^s(X) \rightarrow H^{s-m+n/2}(X)$.
Therefore, it is compact.

Independence of the choice of cut-off functions $\psi$ and $\psi'$ is obvious from calculations.

Thus, $\widetilde{i^!(\Phi)}$ is indeed of the form~\eqref{eq:dualOurSledAsIntegralOperatorSecond}
up to operators which are compact in physical coordinates.
Proposition~\ref{prop:dualOurSledIsIntegralOperator} is proved.
\end{proof}

\subsection{Reduction to the Fourier--Mellin structure}

Now consider $\widetilde{i^!_0(\Phi)}$, that is, a representation of the
reduced trace $i^!_0(\Phi)$ in dual coordinates.

Let us omit the cut-off functions in~\eqref{eq:dualOurSledAsIntegralOperatorSecond} in this section.
Then, instead of $\widetilde{i^!_0(\Phi)}$, we have an operator acting (formally) as follows
\begin{equation}\label{defeq:dualFormalReducedSledOurFIO}
    \widetilde{u}(p) \longmapsto \int {\lvert p \rvert}^{s-m-n/2} {\lvert p' \rvert}^{-s} b^0(p,p') \widetilde{u}(p') \, d{p'}.
\end{equation}
It is an integral operator on $\mathbb{R}^{n/2}$ induced by a kernel
which is a homogeneous function of degree $-n/2$.
Such an operator turns out to be a Mellin convolution in spherical coordinates.

Indeed, let us make in~\eqref{defeq:dualFormalReducedSledOurFIO} the spherical change of coordinates,
\begin{equation*}
    \begin{split}
        & p = r_p\omega_p, \quad r_p \in \mathbb{R}_+, \, \omega_p \in \mathbb{S}^{n/2-1}, \\
        & p' = r_{p'}\omega_{p'}, \quad r_{p'} \in \mathbb{R}_+, \, \omega_{p'} \in \mathbb{S}^{n/2-1}, \\
        & d{p'} = r_{p'}^{n/2} \frac{dr_{p'}}{r_{p'}}d\omega_{p'}.
    \end{split}
\end{equation*}
Then~\eqref{defeq:dualFormalReducedSledOurFIO} takes the form
$$
    \widetilde{u}(p) \longmapsto
    \int_{\mathbb{S}^{n/2-1}}d{\omega_{p'}}
    \int_{\mathbb{R}_+}
        \biggr(\dfrac{r_p}{r_{p'}}\biggl)^{s-m-n/2}
        b^0\biggl(\dfrac{r_p}{r_{p'}}\omega_p,\omega_{p'}\biggr)
        \widetilde{u}(r_{p'}\omega_{p'}) \,\dfrac{dr_{p'}}{r_{p'}}.
$$
Let us introduce an operator-valued function
$K(t)$, $t > 0$, with values in operators acting on sphere $\mathbb{S}^{n/2-1}$ by the formula:
\begin{equation}\label{defeq:K(t)}
    K(t) v(\omega) = \int_{\mathbb{S}^{n/2-1}} t^{s-m-n/2} b^0(t \omega,\omega') v(\omega')\, d{\omega'}.
\end{equation}
Note that the values of $K(t)$ are compact operators in
$L^2(\mathbb{S}^{n/2-1}) \rightarrow L^2(\mathbb{S}^{n/2-1})$,
since they are integral operators with smooth kernels acting on a compact manifold.

Now, operator~\eqref{defeq:dualFormalReducedSledOurFIO} can be written in the form of Mellin convolution
(see, e.g.,~\cite{ConvBook}) with the operator-valued kernel $K(r_p/r_{p'})$:
\begin{equation*}
    \widetilde{i^!_0(\Phi)} \widetilde{u}(p) =
    \int_{\mathbb{R}_+}
        K\biggl(\dfrac{r_p}{r_{p'}}\biggr) \widetilde{u}(r_{p'}\omega_{p'}) \,\dfrac{dr_{p'}}{r_{p'}}.
\end{equation*}
After applying the Mellin transform, this operator becomes an operator acting by multiplication by a function.
More precisely, consider the Mellin transform of $K(t)$:
\begin{equation}\label{defeq:whK}
    \widehat{K}(\zeta)
    =
    \mathcal{M}_{t \rightarrow \zeta} K(t) = \int_{\mathbb{R}_+} t^\zeta K(t) \,\dfrac{dt}{t},
    \quad \operatorname{Re}(\zeta) = \gamma = n/4
\end{equation}
(here the choice of $\gamma$ is made because of~\eqref{eq:weightGaForHs} (see~\cite{Losch1} for details)).
Then $\widetilde{i^!_0(\Phi)}$ takes the form:
$$
    \widetilde{i^!_0(\Phi)} \widetilde{u}(p) =
    \mathcal{M}^{-1}_{\zeta \rightarrow r_p} \widehat{K}(\zeta) \mathcal{M}_{r_p \rightarrow \zeta} \,
        \widetilde{u}(r_{p}\omega_{p}).
$$
Reverting the cut-off function (let us take $\psi' \equiv \psi$), we eventually have
$$
    \widetilde{i^!_0(\Phi)} \widetilde{u}(p) = \psi(p) \, \mathcal{M}^{-1}_{\zeta \rightarrow r_p} \widehat{K}(\zeta) \mathcal{M}_{r_p \rightarrow \zeta} \, \psi(p) \, \widetilde{u}(p).
$$

The function $\widehat{K}(\zeta)$ takes values in compact operators. Let us ensure
that it satisfies the required analytical conditions.

\begin{lemma}
The following assertions hold
\begin{enumerate}
    \item The function $\widehat{K}(\zeta)$ is analytic in the vertical strip
    \begin{equation}\label{neq:polosa}
        -s+m+\frac{n}{2} < \operatorname{Re}(\zeta) < -s+\frac{n}{2}.
    \end{equation}
    \item ${\bigl\| \widehat{K}(\zeta) \bigr\|} \to 0$ as ${\lvert \zeta \rvert} \to \infty$ in strip~\eqref{neq:polosa}.
\end{enumerate}
\end{lemma}
\begin{proof}
1) (see~\cite{SaSt31}).
Apparently, the function $\widehat{K}(\zeta)$ is analytic for all $\zeta = \rho + i \varrho$ such that the integral
$$
    \int_0^\infty t^\rho \left\| K(t) \right\| \,\dfrac{dt}{t}
$$
converges.
Since the function $\psi b^0  \psi'$ is homogeneous at infinity of degree $m$, we can estimate this integral as follows
(see~\eqref{defeq:K(t)})
$$
    \int_0^\infty t^\rho \left\| K(t) \right\| \,\dfrac{dt}{t} \leq C \int_0^\infty t^{\rho+s-m-n/2}(1 + t)^m \,\dfrac{dt}{t}.
$$
The latter integral converges at zero for $\rho > -s+m+n/2$ and at infinity for $\rho < -s+n/2$.
This completes the first assertion.

2) It is sufficient to verify that $\|\widehat{K}(\rho + i \varrho)\|$ tends to zero as $\varrho \rightarrow 0$
for $\rho$ belonging to strip~\eqref{neq:polosa}.
Consider the integral on the right hand side of~\eqref{defeq:whK} and make the substitution $t = e^{-\tau}$.
Then this integral takes the form
\begin{equation*}
    \widehat{K}(\zeta)
    =
    \int_{\mathbb{R}} e^{-i\tau\varrho} \left[e^{-\tau\rho} K(e^{-\tau}) \right] \,d\tau.
\end{equation*}
It is the Fourier transform of a function belonging to $L^1$-space
(by assertion 1)).
Therefore, it decays at infinity.
This completes the second assertion, and thus the lemma is proved.
\end{proof}

Recall again that we have $s < 0$, $s-m-n/2 > 0$ by assumption. It follows that
the line $\operatorname{Re}(\zeta) = \gamma = n/4$ lies inside strip~\eqref{neq:polosa}.
Thus, $\widehat{K}(\zeta)$ possess all the properties that a symbol of a Fourier--Mellin operator
is required to have.

Finally, having moved from dual to physical coordinates, we get an operator
of a Fourier--Mellin type. Note that the possibility of the choice $\gamma = n/4$
provides that it is continuous in $L^2(X) \rightarrow L^2(X)$.

The proof of theorem~\ref{th:main} is finished.

\section{Example}
Here we give a simple calculation for the trace of a canonical transformation represented by a rotation.

Let $M$ be a torus $\mathbb{S}^1 \times \mathbb{S}^1$, and $(x,t)$ be its local coordinates.
Let $X$ be a submanifold given by the equation $t = 0$.
Let $g$ act on $T^*(\mathbb{S}^1 \times \mathbb{S}^1)$ by the following formula:
\begin{equation}
    g: \begin{pmatrix}
         x' \\
         t' \\
         p' \\
         \tau' \\
       \end{pmatrix}
    \longrightarrow
        \begin{pmatrix}
         -t \\
         x \\
         -\tau \\
         p \\
       \end{pmatrix}.
\end{equation}
Note that locally $g$ is a counter-clockwise rotation of $\mathbb{R}^2$ by $\pi/2$,
and $(0,0)$ is the fixed point. Let us focus attention on (a neighbourhood of) this point.

We make quantization of $g$ as follows. Graph
$\operatorname{graph} g$ is represented (locally) as a submanifold in $\mathbb{R}^4$
given by the equation
$$
    \operatorname{graph} g = \{ (x,t,p,\tau; \, x',t',p',\tau') \; | \; x' = -t, \, t'=x, \, p' = -\tau, \, \tau' = p \}.
$$
This is a Lagrangian submanifold, and its generating function has the form
$$
    S(x,t,p',\tau') =\tau'x  - p't.
$$
Take the following amplitude
$$
    a(p', \tau') = \dfrac{1}{(p')^2 + (\tau')^2}.
$$
The operator $\Phi = \Phi(g,a)$ has the form
\begin{equation*}
    \Phi u(x,t) =
    \int_{\mathbb{R}^2}
        \mathcal{F}_{(p',\tau') \rightarrow (x',t')}
        \biggl\{
            e^{i (\tau'x  - p't)} \dfrac{1}{(p')^2 + (\tau')^2}
        \biggr\} \,
        u(x',t')
    \, d{x'}\, d{t'}.
\end{equation*}

Consider its trace $i^!(\Phi)$. By~\eqref{eq:kernelDualSledFIO}, in dual coordinates it takes the form
\begin{equation*}
    \widetilde{i^!(\Phi)} \widetilde{u}(p) =
    \int_{\mathbb{R}} \,
    \biggl[
        \int_{\mathbb{R}^2}
            \mathcal{F}_{(x,t) \rightarrow (p,\tau)}
            \biggl\{
                e^{i (\tau'x  - p't)}
                \dfrac{1}{(p')^2 + (\tau')^2}
            \biggr\}
        \, d{\tau}\, d{\tau'}
    \biggr] \, \widetilde{u}(p') \, d{p'}.
\end{equation*}
Integration with respect to $\tau$ can be replaced by the substitution $t = 0$,
and the Fourier transform of an exponent with a linear phase yields a Dirac delta-function.
Hence,
\begin{equation*}
    \begin{split}
        \widetilde{i^!(\Phi)} \widetilde{u}(p) & =
        \int_{\mathbb{R}} \,
        \biggl[
            \int_{\mathbb{R}}
                \mathcal{F}_{x \rightarrow p}
                \biggl\{
                    e^{i \tau'x}
                    \dfrac{1}{(p')^2 + (\tau')^2}
                \biggr\}
            \, d{\tau'}
        \biggr] \, \widetilde{u}(p') \, d{p'} = \\ & =
        \int_{\mathbb{R}}
        \biggl[
            \int_{\mathbb{R}}
                \delta(p - \tau') \,
                \dfrac{1}{(p')^2 + (\tau')^2}
            \, d{\tau'}
        \biggr] \, \widetilde{u}(p') \, d{p'} = \\ & =
        \int_{\mathbb{R}}
            \dfrac{1}{p^2 + (p')^2} \,
            \widetilde{u}(p')
        \, d{p'}.
    \end{split}
\end{equation*}

Consider the reduced trace. It takes the form (in dual coordinates)
\begin{equation}\label{example:eq:almostMellinConv}
    \begin{split}
        \widetilde{i^!_0(\Phi)} \widetilde{u}(p) & =
        \int_{\mathbb{R}}
            {\lvert p \rvert}^{s+1}\dfrac{{\lvert p \rvert}}{p^2 + (p')^2}
            {\lvert p' \rvert}^{-s}\widetilde{u}(p') \,
        d{p'} = \\ & =
        \int_{\mathbb{R}}
            {\left\lvert {\dfrac{p}{p'}} \right\rvert}^s
            \dfrac{{\lvert p \rvert}{\lvert p' \rvert}}{p^2 + (p')^2} \,
            \widetilde{u}(p')
        \, \dfrac{dp'}{p'}.
    \end{split}
\end{equation}
Splitting $\widetilde{u}(p)$ into odd and even parts and using the fact that the integral kernel
of~\eqref{example:eq:almostMellinConv} is an even function, we obtain
that $\widetilde{i^!(\Phi)}$ can be represented in the following form
\begin{equation*}
    \left(
        \begin{array}{c}
            \widetilde u^+(p) \\
            \widetilde u^-(p) \\
        \end{array}
    \right)
    \longmapsto
    \left(
        \begin{array}{cc}
            2\displaystyle\int_{\mathbb{R}_+}
                {\left\lvert {\dfrac{p}{p'}} \right\rvert}^s
                \dfrac{{\lvert p \rvert}{\lvert p' \rvert}}{p^2 + (p')^2}
                \bullet
            \,\dfrac{dp'}{p'} & 0  \vspace{3mm}\\
            0 & 1 \\
        \end{array}
    \right)
    \left(
        \begin{array}{c}
            \widetilde u^+(p) \\
            \widetilde u^-(p) \\
        \end{array}
    \right),
\end{equation*}
where $\widetilde u^-$ and $\widetilde u^+$
stand for odd and even parts of $\widetilde u$, respectively.

This is a matrix operator.
Note that in the upper-left corner we have a Mellin convolution operator with the function
$$
    K(t) = 2t^s\frac{t}{1+t^2},
$$
After applying the Mellin transform, this matrix operator becomes a multiplication operator by a matrix
\begin{equation*}
    \widehat K(\zeta) =
    \begin{pmatrix}
      \widehat K_{11}(\zeta) & 0 \\
      0 & 1 \\
    \end{pmatrix},
\end{equation*}
where $\widehat K_{11}(\zeta)$ is (see~\cite{BeEr2}, p. 270)
\begin{equation*}
    \begin{split}
        \widehat K_{11}(\zeta) & = \mathcal{M}_{t \rightarrow \zeta}\biggl\{2t^s\frac{t}{1+t}\biggr\} =
        \int_{\mathbb{R}_+} 2t^{s+\zeta} \frac{t}{1+t^2} \,\dfrac{dt}{t} = \\ & =
        \frac{\pi}{\cos\left(\frac{\pi}{2}(\zeta+s)\right)}.
    \end{split}
\end{equation*}


\end{document}